\documentclass[final,3p]{elsarticle}
\usepackage{geometry}
\usepackage{mathrsfs,enumerate,amssymb,amsmath,bm,amsthm,color,lineno}
\usepackage[all]{xy}
\usepackage{latexsym}
\usepackage{amscd,verbatim}
\usepackage{graphicx}
\usepackage[colorlinks=true]{hyperref}
\usepackage{amssymb}
\usepackage{placeins}
\usepackage{float}
\usepackage{varioref}

\numberwithin{equation}{section}

\newtheorem{theorem}{Theorem}
\newtheorem{lemma}{Lemma}
\newtheorem{proposition}{Proposition}
\newtheorem{corollary}{Corollary}

\newtheorem{definition}{Definition}
\newtheorem{example}{Example}
\newtheorem{question}{Question}

\usepackage{pifont}

\begin{document}
\begin{frontmatter}
\title{\huge Ordinal sum of two binary operations being a t-norm on bounded lattice\tnoteref{mytitlenote}}
\tnotetext[mytitlenote]{This work was supported by the National Natural Science Foundation of China
(No. 11601449 and 11701328), the Science and Technology Innovation Team of Education Department of
Sichuan for Dynamical System and its Applications (No. 18TD0013), the Youth Science and Technology
Innovation Team of Southwest Petroleum University for Nonlinear Systems (No. 2017CXTD02), the Key
Natural Science Foundation of Universities in Guangdong Province (No. 2019KZDXM027), Shandong
Provincial Natural Science Foundation, China (Grant ZR2017QA006), and Young Scholars Program
of Shandong University, Weihai (No. 2017WHWLJH09).}


\author[a1,a2,a3]{Xinxing Wu}
\address[a1]{School of Mathematics and Statistics, Guizhou University of Finance and
Economics, Guiyang 550025, China}
\address[a2]{School of Sciences, Southwest Petroleum University, Chengdu, Sichuan 610500, China}
\address[a3]{Zhuhai College of Jilin University, Zhuhai, Guangdong 519041, China}
\ead{wuxinxing5201314@163.com}

\author[a2]{Qin Zhang}
\ead{zhangqin255613@163.com}

\author[a4]{Xu Zhang\corref{mycorrespondingauthor}}
\cortext[mycorrespondingauthor]{Corresponding author}
\address[a4]{Department of Mathematics, Shandong University, Weihai, Shandong 264209, China}
\ead{xu$\_$zhang$\_$sdu@mail.sdu.edu.cn}

\author[a4]{G\"{u}l Deniz \c{C}ayl{\i}}
\address[a4]{Department of Mathematics, Faculty of Science, Karadeniz Technical University, 61080 Trabzon, Turkey}
\ead{guldeniz.cayli@ktu.edu.tr}

\author[a3]{Lidong Wang}
\ead{wld0707@163.com}

%% ÕªÒª
\begin{abstract}
The ordinal sum of t-norms on a bounded lattice has been used to construct
other t-norms. However, an ordinal sum of binary operations (not necessarily
t-norms) defined on the fixed subintervals of a bounded lattice may not be a
t-norm. Some necessary and sufficient conditions are presented in this paper
for ensuring that an ordinal sum on a bounded lattice of two binary
operations is, in fact, a t-norm. In particular, the results presented here
provide an answer to an open problem put forward by Ertu\u{g}rul and Ye\c{s}%
ilyurt [Ordinal sums of triangular norms on bounded lattices, Inf. Sci., 517 (2020) 198-216].
\end{abstract}
%% ¹Ø¼ü´Ê
\end{frontmatter}

\section{Introduction}
Triangular norms (t-norms) were systematically investigated by Schweizer and
Sklar \cite{SS1960,SS1961,SS1983} in the framework of probabilistic metric
spaces aiming at an extension of the triangle inequality. As an extension of
the logical connective \textit{conjunction} in classical two-valued logic,
t-norms have been used widely in many different areas, such as in decision
making \cite{calvo,fodor,grabisch,marichal,walker,KMP2000}, statistics
\cite{nelsen}, fuzzy set theory \cite{nguyen,ray,wu}. Clifford
introduced the notion of ordinal sum on the unit interval $\left[ 0,1\right]
,$ providing a method to produce new t-norms from given ones since the unit
interval together with a t-norm forms a semigroup \cite{C1954}. As a result, a continuous
t-norm can be represented as an ordinal sum of the product t-norm and \L %
ukasiewicz's t-norm \cite{H1998,L1965,MS1957}. Afterward, t-norms were
generalized to more general structures, including posets and bounded
lattices, and their characteristics were extensively investigated \cite%
{BaeMes99,CooKer94,Dro99,Zha05}. Saminger~\cite{S2006} extended the ordinal
sum of t-norms on the unit interval to the ordinal sum of t-norms on
subintervals of a bounded lattice. However, Saminger's definition of ordinal
sum of t-norms on a bounded lattice does not always generate a t-norm. From
this point of view, the constraints were provided in \cite{Med12,SamKleMes08}
to ensure that the ordinal sum of t-norms defined on the subintervals of a
bounded lattice generates a t-norm. Then, some researchers presented some
construction methods for t-norms on a bounded lattice to modify Saminger's
ordinal sum and considered the ordinal sum problem for a particular class of
lattices. In particular, in the papers \cite%
{C2018,C2019,C2020,asici,DHQ2020,holcapek,dvorak,KEK2019,EKM2015,EKK2019},
Saminger's ordinal sum method with one special summand was modified to
guarantee that it is a t-norm on a bounded lattice. El-Zekey~\cite{E2020}
studied the ordinal sum of t-norms on bounded lattices written as a
lattice-based sum of lattices. Dvo\v{r}\'{a}k and Hol\v{c}apek \cite%
{holcapek,dvorak} introduced a new ordinal sum construction of t-norms on
bounded lattices based on interior and closure operators. Ouyang et al. \cite%
{OZB2020} proposed an alternative definition of ordinal sum of countably
many t-norms on subintervals of a complete lattice and proved that it is a
t-norm.

{\color{blue}Recently}, Ertu\u{g}rul and Ye\c{s}ilyurt \cite{EY2020} have extended
the results related to ordinal sum with one summand either based on some
arbitrary or fixed subinterval and dealt with the ordinal sum operations
with more summands such that the ordinal sum on a bounded lattice of
arbitrary t-norms yields again a t-norm. In particular, they have shown that
it is possible to give a construction method to obtain a t-norm on a bounded
lattice $L$ derived from two t-norms $T_{1}$ and $T_{2}$ on the subintervals
$\left[ a,1\right] $ and $\left[ 0,a\right] ,$ respectively, for $a\in
L\setminus \{0,1\}$ without any additional requirement. In the same paper,
they have also proposed an open problem: if we take an associative,
commutative, and monotone binary operation instead of at least one of
t-norms on the subintervals of $L$, will the same method work? If not, what
kind of modification is required? Although there are many results on t-norms
on bounded lattices, their structure is still unclear. Hence, the proposal
in \cite{EY2020} requires further study of t-norms on bounded lattices to
obtain as many of their new classes as possible. In this paper, motivated by
the above-mentioned suggestion, we first demonstrate that the ordinal sum
method with two summands introduced in \cite{EY2020} may not work on a
bounded lattice $L$ when taking $T_{2}$ on the subinterval $\left[ 0,a\right]
$ as an associative, commutative, and increasing binary operation, not
necessarily a t-norm. Then, we look for necessary and sufficient conditions
to ensure that such an ordinal sum method is increasing. Moreover,
considering $T_{1}$ and $T_{2}$ as two t-subnorms on the subintervals $\left[
a,1\right] $ and $\left[ 0,a\right] ,$ respectively, we provide some
necessary and sufficient conditions for the ordinal sum of $T_{1}$ and $%
T_{2} $ being a t-norm on a bounded lattice $L.$ In this way, we give a
complete answer to the above open problem.

The remainder of this paper is organized as follows. In Section \ref{S-2}, we
briefly recall some basic notions and results related to lattices and
t-norms on a bounded lattice. In Section \ref{S-3}, we first review the related
ordinal sum methods on a bounded lattice in the sense of Saminger \cite%
{S2006}, and Ertu\u{g}rul and Ye\c{s}ilyurt \cite{EY2020}. Then, we are
interested in two open problems proposed in \cite{EY2020}. According to their
proposals, Section \ref{S-4} is devoted to investigating whether the ordinal sum
method with two summands on a bounded lattice in the sense of Ertu\u{g}rul
and Ye\c{s}ilyurt \cite{EY2020} is increasing. Here, we consider any binary
operation instead of at least one of two summands being a t-norm on the
fixed subinterval. Some necessary and sufficient conditions are presented in
Section \ref{S-5} for guaranteeing that the ordinal sum method on a bounded lattice
of two t-subnorms is, in fact, a t-norm. We end with some concluding remarks
and future works in Section \ref{S-6}.

\section{Preliminaries}\label{S-2}

In this section, we recall some basic notions and results related to
lattices and t-norms on a bounded lattice.

A \textit{lattice}  is a nonempty set $L$ equipped with a
partial order $\leq $ such that each two elements $x,y\in L$ have a greatest
lower bound, called \textit{meet} or \textit{infimum,} denoted by $x\wedge y$%
, as well as a smallest upper bound, called \textit{join} or \textit{%
supremum,} denoted by $x\vee y$ \cite{Bir1967}. For $x,y\in L$, the symbol $x<y$ means that
$x\leq y$ and $x\neq y$. If $x\leq y$ or $y<x,$ then we say that $x$ and $y$
are comparable. Otherwise, we say that $x$ and $y$ are \textit{incomparable}%
, in this case, we use the notation $x\Vert y$. The set of all elements of $%
L $ that are incomparable with $a$ is denoted by $I_{a}$, i.e., $%
I_{a}=\{x\in L:x\Vert a\}$.

\textit{A bounded lattice} is a lattice $(L,\leq )$ which has the top and
the bottom elements, which are denoted by $1$ and $0$, respectively; that
is, two elements $1,0\in L$ exist such that $0\leq x\leq 1$ for all $x\in L$.

\begin{definition}[\protect\cite{Bir1967}]
{\rm Let $(L,\leq ,0,1)$ be a bounded lattice and $a,b\in L$ with $a\leq b$. The
subinterval $[a,b]$ is defined by
\begin{equation*}
\lbrack a,b]=\{x\in L:a\leq x\leq b\}.
\end{equation*}%
Other subintervals such as $[a,b),$ $\left( a,b\right] $ and $(a,b)$ can be
defined similarly. Obviously, $\left( [a,b],\leq \right) $ is a bounded
lattice with the top element $b$ and the bottom element $a$.}
\end{definition}

Let $(L,\leq ,0,1)$ be a bounded lattice, $[a,b]$ be a subinterval of $L,$ $%
T_{1}$ and $T_{2}$ be two binary operations on $[a,b]^{2}$. If there holds $%
T_{1}(x,y)\leq T_{2}(x,y)$ for all $(x,y)\in \lbrack a,b]^{2}$, then we say
that $T_{1}$ is less than or equal to $T_{2}$ or, equivalently, that $T_{2}$
is greater than or equal to $T_{1}$, and written as $T_{1}\leq T_{2}$.
%Unless otherwise
%specified, all binary operations in this paper are assumed to be weaker than the binary operation `$\wedge$'.

\begin{definition}[\protect\cite{BaeMes99,CooKer94,OZB2020}]
{\rm Let $(L,\leq ,0,1)$ be a bounded lattice and $[a,b]$ be a subinterval of $L$%
. A binary operation $T:[a,b]^{2}\longrightarrow \lbrack a,b]$ is said to be
a \textit{t-norm} on $[a,b]$ if, for any $x,y,z\in \lbrack a,b]$, the
following conditions are fulfilled:
\begin{itemize}
\item[(T$_{1}$)] (\textit{commutativity}) $T(x,y)=T(y,x)$;

\item[(T$_{2}$)] (\textit{associativity}) $T\left( T\left( x,y\right)
,z\right) =T\left( x,T\left( y,z\right) \right) $;

\item[(T$_{3}$)] (\textit{increasingness}) If $x\leq y,$ then $T(x,z)\leq T(y,z$);

\item[(T$_{4}$)] (\textit{neutrality}) $T(b,x)=x$.
\end{itemize}
}
\end{definition}

Notice that $b$ is a neutral element for a t-norm $T:[a,b]^{2}%
\longrightarrow \lbrack a,b]$ while $a$ is a zero element for $T,$ i.e., $%
T(a,x)=a$ for all $x\in \lbrack a,b]$.

\begin{example}
{\rm The following binary operations are examples of t-norms on the subinterval $%
[a,b]$ of a bounded lattice $L$. The \textit{meet} (or \textit{infimum}) t-norm $%
T_{M} $ on $[a,b]$ and the \textit{drastict product t-norm} $T_{D}$ on $[a,b]
$ are defined by%
\begin{equation*}
T_{M}:[a,b]^{2}\longrightarrow \lbrack a,b],\text{ }T_{M}\left( x,y\right)
=x\wedge y,
\end{equation*}%
and%
\begin{equation*}
T_{D}:[a,b]^{2}\longrightarrow \lbrack a,b],\text{ }T_{D}\left( x,y\right) =%
\begin{cases}
x\wedge y, & b\in \left\{ x,y\right\},\\
a, & \text{otherwise}.%
\end{cases}%
\end{equation*}
}
\end{example}

We observe that $T_{M}$ and $T_{D}$, respectively, are the greatest and the
least smallest t-norms on the subinterval $[a,b]$.

\section{Ordinal sums of t-norms on bounded lattices}\label{S-3}

T-norms have been extensively studied on bounded lattices similarly to their
counterparts on the unit interval. Saminger  proposed an ordinal
sum of t-norms defined on some fixed subinterval of a bounded lattice \cite{S2006}. We
can present Saminger's method as follows:

\begin{theorem}[\protect\cite{S2006}]
Let $(L,\leq ,0,1)$ be a bounded lattice and $a\in L\setminus \{0,1\}$. If $%
T_{1}:[a,1]^{2}\longrightarrow \lbrack a,1]$ and $T_{2}:[0,a]^{2}%
\longrightarrow \lbrack 0,a]$ are two t-norms on the subintervals $[a,1]$
and $[0,a]$ of $L,$ respectively, then, the binary operation $T^{\left(
S\right) }:L^{2}\longrightarrow L$ defined by the following formula (\ref%
{eqn1}) is an ordinal sum of t-norms $T_{1}$ and $T_{2}$ on $L.$
\begin{equation}
T^{\left( S\right) }(x,y)=%
\begin{cases}
T_{1}(x,y), & (x,y)\in \left[ a,1\right]^{2}, \\
T_{2}(x,y), & (x,y)\in \left[ 0,a\right]^{2}, \\
x\wedge y, & \text{ otherwise}.%
\end{cases}
\label{eqn1}
\end{equation}
\end{theorem}

According to Saminger \cite{S2006}, however, the above ordinal sum of
t-norms on the fixed subinterval is not always a t-norm. She introduced some
conditions which the bounded lattice $L$ needs to fulfill such that the
ordinal sum method given by the formula (\ref{eqn1}) produces a t-norm on $L$%
.

\begin{theorem}[\protect\cite{S2006}]
Let $(L,\leq ,0,1)$ be a bounded lattice, $a\in L\setminus \{0,1\}$, $%
T_{1}:[a,1]^{2}\longrightarrow \lbrack a,1]$ and $T_{2}:[0,a]^{2}%
\longrightarrow \lbrack 0,a]$ be two t-norms on the subintervals $[a,1]$ and
$[0,a]$ of $L,$ respectively. Then the following statements are equivalent:

\begin{enumerate}
[I)]

\item The ordinal sum $T:L^{2}\longrightarrow L$ of $T_{1}$ and $T_{1}$ defined
by the formula (\ref{eqn1}) is a t-norm on $L$ .

\item For all $x\in L$, it holds that

\begin{itemize}
\item[\textrm{II-1)}] if $x$ is incomparable with $a$, then it is
incomparable to all element in $[a,1)$,

\item[\textrm{II-2)}] if $x$ is incomparable with $a$, then it is
incomparable to all elements in $(0,a]$.
\end{itemize}
\end{enumerate}
\end{theorem}

Several researchers  characterized when Saminger's
ordinal sum of t-norms always yield a t-norm on a bounded lattice \cite{Med12,SamKleMes08}, while
other researchers attempted to
modify Saminger's ordinal sum or considered the ordinal sum problem for a
particular class of lattices \cite{C2018,C2019,C2020,EKM2015,EKK2019}. In recent times, Ertu\v{g}rul and Ye\c{s}%
ilyurt have introduced an ordinal sum method to produce a
t-norm on a bounded lattice $L$ by using the t-norms defined on the
indicated subintervals of $L$~\cite{EY2020}. Their ordinal sum method is presented in the
following Theorem \ref{EY-Thm}.

\begin{theorem}[\protect\cite{EY2020}]
\label{EY-Thm} Let $(L,\leq ,0,1)$ be a bounded lattice and $a\in L\setminus
\{0,1\}$. If $T_{1}:[a,1]^{2}\longrightarrow \lbrack a,1]$ and $%
T_{2}:[0,a]^{2}\longrightarrow \lbrack 0,a]$ are two t-norms on the
subintervals $[a,1]$ and $[0,a]$ of $L,$ respectively, then, the binary
operation $T:L^{2}\longrightarrow L$ defined by the following formula (\ref%
{*}) is a t-norm on $L.$
\begin{equation}
T(x,y)=%
\begin{cases}
T_{1}(x,y), & (x,y)\in \lbrack a,1)^{2}, \\
T_{2}(x,y), & (x,y)\in \lbrack 0,a)^{2}, \\
x\wedge y, & (x,y)\in \lbrack 0,a)\times \lbrack a,1)\\
& \quad\quad \cup \left(\lbrack a,1)\times \lbrack 0,a)\right) \\
& \quad\quad \cup \left(L\times \{1\}\right)\\
& \quad\quad \cup \left(\{1\}\times L\right), \\
T_{2}(x\wedge a,y\wedge a), & \text{otherwise}.%
\end{cases}
\label{*}
\end{equation}
\end{theorem}

By this, the following Question \ref{Q-1} arises quite naturally, which has
been proposed as an open problem by Ertu\u{g}rul and Ye\c{s}ilyurt \cite%
{EY2020}.

\begin{question}[\protect\cite{EY2020}]
\label{Q-1} {\rm Given a bounded lattice $(L,\leq ,0,1)$ and $a\in L\setminus
\{0,1\},$ if we take in Theorem \ref{EY-Thm} an associative, commutative and
monotone binary operation instead of at least one of the t-norms $T_{1}$ and
$T_{2}$ defined on the subintervals $[a,1]$ and $[0,a]$ of $L$,
respectively, does the function $T$ given by the formula (\ref{*}) need to
yield a t-norm on $L$?}
\end{question}

We introduce the following Proposition \ref{necessary-cond} to help us
answer this question.

\begin{proposition}
\label{necessary-cond} Let $(L,\leq ,0,1)$ be a bounded lattice, $a\in
L\setminus \{0,1\}$, $T_{1}:[a,1]^{2}\longrightarrow \lbrack a,1]$ and $%
T_{2}:[0,a]^{2}\longrightarrow \lbrack 0,a]$ be two binary operations on the
subintervals $[a,1]$ and $[0,a]$ of $L,$ respectively. If the function $%
T:L^{2}\longrightarrow L$ defined by the formula (\ref{*}) is increasing,
then $T_{1}(x_{1},y_{1})\leq x_{1}\wedge y_{1}$ and $%
T_{2}(x_{2},y_{2})\leq x_{2}\wedge y_{2}$ for any $(x_{1},y_{1})\in \lbrack
a,1)^{2}$ and $(x_{2},y_{2})\in \lbrack 0,a)^{2}$ hold.
\end{proposition}

\begin{proof}
From the increasingness and the definition of $T$, it follows that for any $%
(x_{1},y_{1})\in \lbrack a,1)^{2},$ $T_{1}\left( x_{1},y_{1}\right) =T\left(
x_{1},y_{1}\right) \leq T\left( x_{1},1\right) =x_{1}$ and $T_{1}\left(
x_{1},y_{1}\right) =T\left( x_{1},y_{1}\right) \leq T\left( 1,y_{1}\right)
=y_{1}$. Hence, we have $T_{1}(x_{1},y_{1})\leq x_{1}\wedge y_{1}$ for any $%
(x_{1},y_{1})\in \lbrack a,1)^{2}.$

Similarly, it is shown that $T_{2}(x_{2},y_{2})\leq x_{2}\wedge y_{2}$ for
any $(x_{2},y_{2})\in \lbrack 0,a)^{2}.$
\end{proof}

\medskip

By using Proposition \ref{necessary-cond}, we provide the following Example %
\ref{Exa-1} answering negatively to Question \ref{Q-1}. We first take in
Theorem~\ref{EY-Thm} an associative, commutative, and monotone binary
operation instead of at least one of the t-norms on the subintervals of a
bounded lattice $L.$ Then, we show that the function $T$ given by the
formula (\ref{*}) is not a t-norm on $L$.

\begin{example}
\label{Exa-1} {\rm Let $L=[0,1]$ and $a=\frac{1}{2}$. Define $T_{1}:[\frac{1}{2}%
,1]^{2}\longrightarrow \lbrack \frac{1}{2},1]$ and $T_{2}:[0,\frac{1}{2}%
]^{2}\longrightarrow \lbrack 0,\frac{1}{2}]$ by $T_{1}(x,y)=x\wedge y$ and $%
T_{2}(x,y)=\frac{1}{2}$ for all $x,y\in \lbrack 0,\frac{1}{2}]$,
respectively. It is easy to observe that $T_{1}$ is a t-norm on $[\frac{1}{2}%
,1],$ and $T_{2}$ is an associative, commutative, and monotone binary
operation on $[0,\frac{1}{2}]$. It follows from Proposition~\ref%
{necessary-cond} that $T:[0,1]^{2}\longrightarrow \lbrack 0,1]$ given by the
formula (\ref{*}) is not increasing, i.e., $T$ is not a t-norm. Because, if
we take $x=\frac{1}{4}\leq y=\frac{2}{3}$ and $z=\frac{1}{4}$, we obtain $%
T(x,z)=T(\frac{1}{4},\frac{1}{4})={T_{2}}(\frac{1}{4},\frac{1}{4})=\frac{1}{2%
}>\frac{1}{4}=\frac{2}{3}\wedge \frac{1}{4}=T(\frac{2}{3},\frac{1}{4}%
)=T(y,z) $.}
\end{example}

In the paper \cite{EY2020}, it has also been posed another open problem
introduced in Question \ref{Q-2}.

\begin{question}[\protect\cite{EY2020}]
\label{Q-2} {\rm If the function $T$ given by the formula (\ref{*}) does not need
to be a t-norm on a bounded lattice $L$ when considering an associative,
commutative and monotone binary operation instead of at least one of the
t-norms on the subintervals of $L$ in Theorem~\ref{EY-Thm}, what kind of
modification is required?}
\end{question}

We are motivated in this paper from Example \ref{Exa-1} that gives a
negative answer to Question~\ref{Q-1}. We aim to present a sufficient and
necessary condition for ensuring that the ordinal sum of two families of
binary operations on the subintervals of a bounded lattice $L$ defined by
the formula (\ref{*}) is always a t-norm on $L$. In particular, this paper,
including Question~\ref{Q-2}, obtains some types of characterizations for
the ordinal sum defined by the formula (\ref{*}) being an increasing binary
operation in Section \ref{S-3} and a t-norm in Section \ref{S-4}. These characterizations
exactly solve Question~\ref{Q-2} and show that the ordinal sum defined by
the formula (\ref{*}) is closely related to the boundary values of the
binary operation on $((\{a\}\cup I_{a})\times \{a\})\cup ((\{a\}\cup
I_{a})\times \{a\}) $ for $a\in L\setminus \{0,1\}$.

\section{Ordinal sum operation with two summands being increasing}\label{S-4}

In this section, we concentrate on conditions that the operation $%
T:L^{2}\longrightarrow L$ defined by the formula (\ref{*}) is increasing with respect to both
variables on a bounded lattice $L.$

The following Theorem \ref{Increasingness-Thm} provides a partial answer to
Question~\ref{Q-2}. To put a finer point on it, we present a sufficient and
necessary condition for the ordinal sum $T:L^{2}\longrightarrow L$ of two
commutative operations $T_{1}$ and $T_{2}$ defined by the formula (\ref{*})
being an increasing operation on a bounded lattice $L$.

\begin{lemma}
\label{incom-lemma} Let $(L,\leq ,0,1)$ be a bounded lattice and $a\in
L\setminus \{0,1\}.$ Then, there holds $x\wedge a<a$ for any $x\in I_{a}.$
\end{lemma}

The result is proved straightforwardly.

\begin{theorem}
\label{Increasingness-Thm} Let $(L,\leq ,0,1)$ be a bounded lattice and $%
a\in L\setminus \{0,1\}$. Assume that $T_{1}:[a,1]^{2}\longrightarrow
\lbrack a,1]$ and $T_{2}:[0,a]^{2}\longrightarrow \lbrack 0,a]$ are two
commutative binary operations on the subintervals $[a,1]$ and $[0,a]$ of $L,$
respectively, where $T_{1}\left( x,y\right) \leq x\wedge y$ for all $x,y\in %
\left[ a,1\right] $ and $T_{2}\left( x,y\right) \leq x\wedge y$ for all $%
x,y\in \left[ 0,a\right] .$ Then the following statements are equivalent:

\begin{enumerate}
[I)]

\item \label{e-1.1} The binary operation $T:L^{2}\longrightarrow L$ defined by
the formula (\ref{*}) is increasing with respect to both variables on $L$.

\item \label{e-2.1} The following hold:

\begin{itemize}
\item[\textrm{II-1)}] $T_{1}$ and $T_{2}$ are increasing with respect to
both variables on the subintervals $[a,1)$ and $[0,a)$ of $L,$ respectively.

\item[\textrm{II-2)}] $\left\{ x\in L:\text{ }x\in I_{a}\text{ and }%
T_{2}(x\wedge a,a)<x\wedge a\right\} =\emptyset .$
\end{itemize}

\item \label{e-3.1} The following hold:

\begin{itemize}
\item[\textrm{III-1)}] $T_{1}$ and $T_{2}$ are increasing with respect to
both variables on the subintervals $[a,1)$ and $[0,a)$ of $L,$ respectively.

\item[\textrm{III-2)}] $I_{a}=\emptyset $ or $T_{2}(z\wedge a,a)=z\wedge a$
for all $z\in I_{a}$.
\end{itemize}
\end{enumerate}
\end{theorem}

\begin{proof}
(\ref{e-1.1})$\Longrightarrow $(\ref{e-2.1}).

\medskip

II-1) It is obtained straightforwardly.

II-2) On the contrary, suppose that $\left\{ x\in L:\text{ }x\in I_{a}\text{
and }T_{2}(x\wedge a,a)<x\wedge a\right\} \neq \emptyset $. This implies
that there exists an element $y\in I_{a}$ such that $T_{2}(y\wedge
a,a)<y\wedge a$. By Lemma~\ref{incom-lemma}, it is obvious that $a\wedge
y\in \lbrack 0,a)$. Then, we have
\begin{equation}
T(y,a)=T_{2}(y\wedge a,a)<y\wedge a=(y\wedge a)\wedge a=T(y\wedge a,a).
\label{eq-3.1.1}
\end{equation}%
By the increasingness of $T,$ for $a\wedge y\leq y,$ we have $T(a\wedge
y,a)\leq T(y,a).$ This is a contradiction with the formula (\ref{eq-3.1.1}).
Hence, it holds that $\left\{ x\in L:\text{ }x\in I_{a}\text{ and }%
T_{2}(x\wedge a,a)<x\wedge a\right\} =\emptyset $.

\medskip

(\ref{e-2.1})$\Longrightarrow $(\ref{e-3.1}).

\medskip

III-1) It is obtained straightforwardly.

III-2) Since $\left\{ x\in L:\text{ }x\in I_{a}\text{ and }T_{2}(x\wedge
a,a)<x\wedge a\right\} =\emptyset ,$ then $I_{a}=\emptyset $ or $%
T_{2}(z\wedge a,a)\nless z\wedge a$ for any $z\in I_{a}.$ If $%
I_{a}=\emptyset ,$ the proof is completed. Suppose that $I_{a}\neq \emptyset
.$ Then, $T_{2}(z\wedge a,a)\nless z\wedge a.$ From the fact that the binary
operation $T_{2}:[0,a]^{2}\longrightarrow \lbrack 0,a],$ $T_{2}\left(
x,y\right) \leq x\wedge y$ is assumed, it follows $T_{2}(a\wedge z,a)\leq
a\wedge z.$ Then, we obtain $T_{2}(z\wedge a,a)=z\wedge a.$

\medskip

(\ref{e-3.1})$\Longrightarrow $ (\ref{e-1.1}).

\medskip

From the assumption of the binary operation $T_{2}:[0,a]^{2}\longrightarrow
\lbrack 0,a],$ $T_{2}\left( x,y\right) \leq x\wedge y$ and the definition of
$T$, the commutativity of $T$ and the fact that $1$ and $0$ are the neutral
and zero elements of $T$, respectively, are evident. We demonstrate the
increasingness of $T.$ The proof is split into all remaining possible cases.

\medskip

\textit{Claim 1.} $T(x,y)\leq x\wedge y$ for any $x,y\in L$.

\medskip

Considering the commutativity of $T$, it is sufficient to check only that $%
T(x,y)\leq x$.

1-1) If $x=1$, it is clear that $T(1,y)=y\leq 1$.

1-2) If $x\in \lbrack 0,a)$, then by the assumption that $%
T_{2}:[0,a]^{2}\longrightarrow \lbrack 0,a],$ $T_{2}\left( x,y\right) \leq
x\wedge y$, it is verified that%
\begin{align*}
T(x,y)& =%
\begin{cases}
T_{2}(x,y), & y\in \lbrack 0,a), \\
x\wedge y, & y\in \lbrack a,1), \\
x, & y=1, \\
T_{2}(x,y\wedge a), & y\in I_{a},%
\end{cases}
\\
& \leq
\begin{cases}
x\wedge y, & y\in \lbrack 0,a), \\
x, & y\in \lbrack a,1), \\
x, & y=1, \\
x\wedge y\wedge a, & y\in I_{a},%
\end{cases}
\\
& \leq y.
\end{align*}

1-3) If $x\in \lbrack a,1)$, then by the assumptions that $%
T_{2}:[0,a]^{2}\longrightarrow \lbrack 0,a],$ $T_{2}\left( x,y\right) \leq
x\wedge y$ and $T_{1}:[a,1]^{2}\longrightarrow \lbrack a,1],T_{1}\left(
x,y\right) \leq x\wedge y$, it is verified that
\begin{align*}
T(x,y)& =%
\begin{cases}
x\wedge y, & y\in \lbrack 0,a), \\
T_{1}(x,y), & y\in \lbrack a,1), \\
x, & y=1, \\
T_{2}(a,y\wedge a), & y\in I_{a},%
\end{cases}
\\
& \leq
\begin{cases}
x\wedge y, & y\in \lbrack 0,a), \\
x\wedge y, & y\in \lbrack a,1), \\
x, & y=1, \\
y\wedge a, & y\in I_{a},%
\end{cases}
\\
& \leq y.
\end{align*}

1-4) If $x\in I_{a}$, then by the assumption that $T_{2}:[0,a]^{2}%
\longrightarrow \lbrack 0,a],$ $T_{2}\left( x,y\right) \leq x\wedge y$, it
is verified that
\begin{align*}
T(x,y)& =%
\begin{cases}
x, & y=1, \\
T_{2}(x\wedge a,y\wedge a), & \text{otherwise},%
\end{cases}
\\
& \leq
\begin{cases}
x, & y=1, \\
x\wedge a\wedge y, & \text{otherwise},%
\end{cases}
\\
& \leq y.
\end{align*}

\textit{Claim 2. }$T(x,z)\leq T(y,z)$ for all $x,y,z\in L$ with $x\leq y$.

2-1) If $x=1$, then $y=1$. This implies that $T(x,z)=z\leq z=T(y,z)$.

2-2) If $z=1$, then $T(x,z)=x\leq y=T(y,z)$.

2-3) If $y=1$, then by applying Claim 1, $T(x,z)\leq x\wedge z\leq z=T(y,z)$.

2-4) If $x,y\in \lbrack 0,a)$, it is verified that
\begin{equation*}
T(x,z)=%
\begin{cases}
T_{2}(x,z), & z\in \lbrack 0,a), \\
x\wedge z, & z\in \lbrack a,1), \\
T_{2}(x,z\wedge a), & z\in I_{a},%
\end{cases}%
\end{equation*}%
and
\begin{equation*}
T(y,z)=%
\begin{cases}
T_{2}(y,z), & z\in \lbrack 0,a), \\
y\wedge z, & z\in \lbrack a,1), \\
T_{2}(y,z\wedge a), & z\in I_{a}.%
\end{cases}%
\end{equation*}%
These, together with (III-1) and the fact that $z\wedge a<a$ for $z\in I_{a}$%
, imply that $T(x,z)\leq T(y,z)$.

2-5). If $x,y\in \lbrack a,1)$, it is verified that
\begin{equation*}
T(x,z)=%
\begin{cases}
x\wedge z, & z\in \lbrack 0,a), \\
T_{1}(x,z), & z\in \lbrack a,1), \\
T_{2}(a,z\wedge a), & z\in I_{a},%
\end{cases}%
\end{equation*}%
and
\begin{equation*}
T(y,z)=%
\begin{cases}
y\wedge z, & z\in \lbrack 0,a), \\
T_{1}(y,z), & z\in \lbrack a,1), \\
T_{2}(a,z\wedge a), & z\in I_{a}.%
\end{cases}%
\end{equation*}%
These, together with (III-1) and the fact that $z\wedge a<a$ for $z\in I_{a}$%
, imply that $T(x,z)\leq T(y,z)$.

2-6) If $x,y\in I_{a}$, then $x\wedge a<a$ and $y\wedge a<a$. By (III-2), it
is verified that%
\begin{align*}
&T(x,z)\\
=&
\begin{cases}
T_{2}(x\wedge a,z\wedge a), & z\in \{z_{1}\in L\setminus \{1\}:z_{1}\wedge
a<a\}, \\
x\wedge a, & z\in \{z_{1}\in L\setminus \{1\}:z_{1}\wedge a=a\},%
\end{cases}%
\end{align*}%
and
\begin{align*}
&T(y,z)\\
=&\begin{cases}
T_{2}(y\wedge a,z\wedge a), & z\in \{z_{1}\in L\setminus \{1\}:z_{1}\wedge
a<a\}, \\
y\wedge a, & z\in \{z_{1}\in L\setminus \{1\}:z_{1}\wedge a=a\}.%
\end{cases}%
\end{align*}%
These, together with (III-1), imply that $T(x,z)\leq T(y,z)$.

2-7) If $x\in \lbrack 0,a)$ and $y\in I_{a}$, then $y\wedge a<a.$ By
(III-2), it is verified that
\begin{align*}
T(x,z)=%
\begin{cases}
T_{2}(x,z), & z\in \lbrack 0,a), \\
x, & z\in \lbrack a,1), \\
T_{2}(x,z\wedge a), & z\in I_{a},%
\end{cases}%
\end{align*}%
and%
\begin{align*}
&T(y,z)\\
=&%
\begin{cases}
T_{2}(y\wedge a,z\wedge a), & z\in \{z_{1}\in L\setminus \{1\}:z_{1}\wedge
a<a\}, \\
y\wedge a, & z\in \{z_{1}\in L\setminus \{1\}:z_{1}\wedge a=a\},%
\end{cases}%
\end{align*}%
These, together with (III-1) and the fact that $z\wedge a<a$ for $z\in I_{a}$%
, imply that $T(x,z)\leq T(y,z)$.

2-8) If $x\in I_{a}$ and $y\in \lbrack a,1),$ then $x\wedge a<a.$ By (III-2)
and the assumption that $T_{2}:[0,a]^{2}\longrightarrow \lbrack 0,a],$ $%
T_{2}\left( x,y\right) \leq x\wedge y$, it is verified that%
\begin{align*}
&T(x,z)\\
 =&%
\begin{cases}
T_{2}(x\wedge a,z\wedge a), & z\in \{z_{1}\in L\setminus \{1\}:z_{1}\wedge
a<a\}, \\
x\wedge a, & z\in \{z_{1}\in L\setminus \{1\}:z_{1}\wedge a=a\},%
\end{cases}
\\
 \leq&
\begin{cases}
x\wedge z\wedge a, & z\in \{z_{1}\in L\setminus \{1\}:z_{1}\wedge a<a\}, \\
x\wedge a, & z\in \{z_{1}\in L\setminus \{1\}:z_{1}\wedge a=a\},%
\end{cases}%
\end{align*}%
and%
\begin{align*}
T(y,z)& =%
\begin{cases}
y\wedge z, & z\in \lbrack 0,a), \\
T_{1}(y,z), & z\in \lbrack a,1), \\
T_{2}(a,z\wedge a), & z\in I_{a},%
\end{cases}
\\
& \geq
\begin{cases}
y\wedge z, & z\in \lbrack 0,a), \\
a, & z\in \lbrack a,1), \\
z\wedge a, & z\in I_{a}.%
\end{cases}%
\end{align*}

These, together with (III-1) and the fact that $z\wedge a<a$ for $z\in I_{a}$%
, imply that $T(x,z)\leq T(y,z)$.

2-9) If $x\in \lbrack 0,a)$ and $y\in \lbrack a,1)$, then by (III-2) and the
assumption that $T_{2}:[0,a]^{2}\longrightarrow \lbrack 0,a],$ $T_{2}\left(
x,y\right) \leq x\wedge y$, it is verified that
\begin{align*}
T(x,z)& =%
\begin{cases}
T_{2}(x,z), & z\in \lbrack 0,a), \\
x, & z\in \lbrack a,1), \\
T_{2}(x,z\wedge a), & z\in I_{a},%
\end{cases}
\\
& \leq
\begin{cases}
x\wedge z, & z\in \lbrack 0,a), \\
a, & z\in \lbrack a,1), \\
x\wedge z\wedge a, & z\in I_{a},%
\end{cases}%
\end{align*}%
and
\begin{align*}
T(y,z)& =%
\begin{cases}
z, & z\in \lbrack 0,a), \\
T_{1}(y,z), & z\in \lbrack a,1), \\
T_{2}(a,z\wedge a), & z\in I_{a},%
\end{cases}
\\
& \geq
\begin{cases}
z, & z\in \lbrack 0,a), \\
a, & z\in \lbrack a,1), \\
z\wedge a, & z\in I_{a}.%
\end{cases}%
\end{align*}%
These imply that $T(x,z)\leq T(y,z)$.

Thereofere, we obtain that the binary operation $T:L^{2}\longrightarrow L$
defined by the formula (\ref{*}) is increasing with respect to both
variables on $L$.
\end{proof}

\medskip

By Theorem~\ref{Increasingness-Thm}, the increasingness of the binary
operation $T:L^{2}\longrightarrow L$ defined by the formula (\ref{*}) is
equivalent to the increasingness of the binary operations $T_{1}$ on $[a,1)$
and $T_{2}$ on $[0,a)$ (excluding the right endpoint). This means that the
ordinal sum of two commutative binary operations $T_{1}$ and $T_{2}$ given
by the formula (\ref{*}) satisfies the increasingness, commutativity, and
neutrality properties, where $T_{1}:[a,1]^{2}\longrightarrow \lbrack
a,1],T_{1}\left( x,y\right) \leq x\wedge y$ and $T_{2}:[0,a]^{2}%
\longrightarrow \lbrack 0,a],$ $T_{2}\left( x,y\right) \leq x\wedge y$ under
the assumptions (III-1) and (III-2). This ordinal sum operation is not
interested in the values of $T_{1}$ and $T_{2}$ on the boundary $%
(\{a\}\times ([0,a]\setminus \{z\wedge a:z\in I_{a}\}))\cup (([0,a]\setminus
\{z\wedge a:z\in I_{a}\})\times \{a\})$ and $(\{1\}\times \lbrack a,1])\cup
([a,1]\times \{1\})$, respectively.

\section{Ordinal sum operation with two summands being a t-norm}\label{S-5}

Theorem \ref{Increasingness-Thm} deals with an ordinal sum of two
commutative binary operations on the fixed subintervals being increasing on
a bounded lattice $L.$ In this section, we extend this result to ordinal
sums with two summands yielding a t-norm on a bounded lattice $L.$ To be
more precise, we focus on an ordinal sum of t-norms on $L$ built from two
binary operations defined on the subintervals $[0,a]$ and $[a,1]$ for $a\in
L\setminus \{0,1\}.$

Recently, El-Zekey  has extended the concept of \textit{%
t-subnorm} on the unit interval introduced by Klement et al. (\cite{KMP2000})
to lattices \cite{E2020}. In the following, we define t-subnorms on bounded lattices,
which are used in the sequel.

\begin{definition}[\protect\cite{E2020}]
{\rm Let $(L,\leq ,0,1)$ be a bounded lattice and $[a,b]$ be a subinterval of $L$%
. A binary operation $F:[a,b]^{2}\longrightarrow \lbrack a,b]$ is said to be
a \textit{t-subnorm} on $[a,b]$ if it is commutative, associative,
increasing in both arguments and it satisfies the range condition $%
F(x,y)\leq x\wedge y$ for all $x,y\in L$.}
\end{definition}

Every t-norm is a t-subnorm. Nevertheless, the contrary of this argument
does not need to true. For example, given the trivial t-subnorm defined by $%
F(x,y)=0$ for all $x,y\in L$, it is not a t-norm since there is no neutral
element of $F$. It should be pointed out that, like any t-norm, any
t-subnorm $F$ on a bounded lattice $L$ has the zero element $0,$ i.e., $%
F(0,x)=0$ for all $x\in L$.

We now investigate which types of binary operations defined on the fixed
subinterval of $L$ are appropriate candidates for the ordinal sum with two
summands given by the formula (\ref{*}), guaranteeing that it yields a
t-norm on a bounded lattice $L$. The following Theorem \ref{t-norm-Thm}
provides a sufficient and necessary condition for the ordinal sum $%
T:L^{2}\longrightarrow L$ of two t-subnorms on subintervals of $L$ given by the
formula (\ref{*}) being a t-norm on a bounded lattice $L$. In this way, we
provide a complete answer to Question~\ref{Q-2}.

\begin{theorem}
\label{t-norm-Thm} Let $(L,\leq ,0,1)$ be a bounded lattice, $a\in
L\setminus \{0,1\}$, $T_{1}:[a,1]^{2}\longrightarrow \lbrack a,1]$ and $%
T_{2}:[0,a]^{2}\longrightarrow \lbrack 0,a]$ be two t-subnorms on the
subintervals $[a,1]$ and $[0,a]$ of $L,$ respectively. The following
statements are equivalent:

\begin{enumerate}
[I)]

\item \label{a} The binary operation $T:L^{2}\longrightarrow L$ defined by the
formula (\ref{*}) is a t-norm on $L$.

\item \label{b} $\left\{ x\in L:\text{ }x\in I_{a}\text{ and }T_{2}(x\wedge
a,a)<x\wedge a\right\} =\emptyset .$

\item \label{c} $I_{a}=\emptyset $ or $T_{2}(z\wedge a,a)=z\wedge a$ for all
$z\in I_{a}$.
\end{enumerate}
\end{theorem}

\begin{proof}
(\ref{a}) $\Longrightarrow $ (\ref{b}) $\Longrightarrow $ (\ref{c}).

\medskip

It follows from Theorem~\ref{Increasingness-Thm}.

\medskip

(\ref{c}) $\Longrightarrow $ (\ref{a}).

\medskip

Let $I_{a}=\emptyset $ or $T_{2}(z\wedge a,a)=z\wedge a$ for all $z\in I_{a}$%
. It follows from Theorem~\ref{Increasingness-Thm} that the commutativity and
increasingness of $T$ hold. We demonstrate the associativity of $T$. Taking into
account of the commutativity of $T,$ the proof is split into all remaining
possible cases.

\medskip

\textit{Claim: }$T(x,T(y,z))=T(T(x,y),z)$ for any $x,y,z\in L$.

\medskip

1) If one of the elements $x$, $y$ and $z$ is equal to $1$, the equality
holds.

2) Let $x,y\in \lbrack 0,a).$

2-1) If $z\in \lbrack 0,a)$,
\begin{align*}
&T(x,T(y,z)) \\
=&T(x,T_{2}(y,z))=T_{2}(x,T_{2}(y,z)) \\
=&T_{2}(T_{2}(x,y),z)=T(T_{2}(x,y),z)=T(T(x,y),z).
\end{align*}

2-2) If $z\in \lbrack a,1)$,
\begin{align*}
&T(x,T(y,z))\\
=&T(x,y)=T_{2}(x,y)=T(T_{2}(x,y),z)=T(T(x,y),z).
\end{align*}

2-3) If $z\in I_{a}$,
\begin{align*}
&T(x,T(y,z))\\
 =&T(x,T_{2}(y,z\wedge a))=T_{2}(x,T_{2}(y,z\wedge a)) \\
=&T_{2}(T_{2}(x,y),z\wedge a)=T(T_{2}(x,y),z)=T(T(x,y),z).
\end{align*}

3) Let $x\in \lbrack 0,a)$ and $y\in \lbrack a,1).$

3-1) If $z\in \lbrack 0,a)$,
\begin{align*}
&T(x,T(y,z))\\
=&T(x,z)=T_{2}(x,z)=T(x,z)=T(T(x,y),z).
\end{align*}

3-2) If $z\in \lbrack a,1)$,
\begin{align*}
&T(x,T(y,z))\\
=&T(x,T_{1}(y,z))=x=T(x,z)=T(T(x,y),z).
\end{align*}

3-3) If $z\in I_{a}$,%
\begin{align*}
&T(x,T(y,z))\\
=&T(x,T_{2}(a,z\wedge a))=T(x,z\wedge a)~\\
=&T_{2}(x,z\wedge a)~~(%
\text{by (\ref{c})})
\end{align*}%
and%
\begin{equation*}
T(T(x,y),z)=T(x,z)=T_{2}(x,z\wedge a),
\end{equation*}%
which imply that $T(x,T(y,z))=T(T(x,y),z).$

4) Let $x\in \lbrack 0,a)$ and $y\in I_{a}.$

4-1) If $z\in \lbrack 0,a)$,
\begin{align*}
&T(x,T(y,z)) \\
=&T(x,T_{2}(y\wedge a,z))=T_{2}(x,T_{2}(y\wedge a,z)) \\
=&T_{2}(T_{2}(x,y\wedge a),z)=T(T_{2}(x,y\wedge a),z)\\
=&T(T(x,y),z).
\end{align*}

4-2) If $z\in \lbrack a,1)$,
\begin{align*}
&T(x,T(y,z))\\
=&T(x,T_{2}(y\wedge a,a))=T(x,y\wedge a)\\
=&T_{2}(x,y\wedge a)~~%
\text{(by (\ref{c}))}
\end{align*}%
and%
\begin{equation*}
T(T(x,y),z)=T(T_{2}(x,y\wedge a),z)=T_{2}(x,y\wedge a)
\end{equation*}%
which imply that $T(x,T(y,z))=T(T(x,y),z).$

4-3) If $z\in I_{a}$,%
\begin{eqnarray*}
&T(x,T(y,z)) =T(x,T_{2}(y\wedge a,z\wedge a))\\
=&T_{2}(x,T_{2}(y\wedge
a,z\wedge a))=T_{2}(T_{2}(x,y\wedge a),z\wedge a)\\
=&T(T_{2}(x,y\wedge a),z)=T(T(x,y),z).
\end{eqnarray*}

5) Let $x,y\in \lbrack a,1)$ and $z\in I_{a}$.
\begin{equation*}
T(x,T(y,z))=T(x,T_{2}(a,z\wedge a))=T(x,z\wedge a)=z\wedge a\text{ \ }(\text{%
by (\ref{c})})
\end{equation*}%
and
\begin{equation*}
T(T(x,y),z)=T(T_{1}(x,y),z)=T_{2}(a,z\wedge a)=z\wedge a~~(\text{by (\ref{c})%
}),
\end{equation*}%
which imply that $T(x,T(y,z))=T(T(x,y),z)$.

6) Let $x\in \lbrack a,1)$ and $y\in I_{a}.$

6-1) If $z\in \lbrack 0,a)$,
\begin{equation*}
T(x,T(y,z))=T(x,T_{2}(y\wedge a,z))=T_{2}(y\wedge a,z)
\end{equation*}%
and%
\begin{equation*}
T(T(x,y),z)=T(T_{2}(a,y\wedge a),z)=T_{2}(y\wedge a,z)\text{ \ }(\text{by (%
\ref{c})}),
\end{equation*}%
which imply that $T(x,T(y,z))=T(T(x,y),z)$.

6-2) If $z\in \lbrack a,1)$,
\begin{equation*}
T(x,T(y,z))=T(x,T_{2}(y\wedge a,a))=T(x,y\wedge a)=y\wedge a\text{ \ }(\text{%
by (\ref{c})}),
\end{equation*}%
and
\begin{equation*}
T(T(x,y),z)=T(T_{2}(a,y\wedge a),z)=T_{2}(y\wedge a,z)=y\wedge a~~(\text{by (%
\ref{c})}),
\end{equation*}%
which imply that $T(x,T(y,z))=T(T(x,y),z)$.

6-3) If $z\in I_{a}$,
\begin{equation*}
T(x,T(y,z))=T(x,T_{2}(y\wedge a,z\wedge a))=T_{2}(y\wedge a,z\wedge a)\text{
}
\end{equation*}%
and
\begin{align*}
&T(T(x,y),z)\\
=&T(T_{2}(a,y\wedge a),z)=T_{2}(y\wedge a,z)\\
=&T_{2}(y\wedge
a,z\wedge a)\text{ }~(\text{by (\ref{c})}),
\end{align*}%
which imply that $T(x,T(y,z))=T(T(x,y),z)$.

7) Let $x,y,z\in I_{a}$.%
\begin{align*}
&T(x,T(y,z)) \\
=&T(x,T_{2}(y\wedge a,z\wedge a))=T_{2}(x\wedge a,T_{2}(y\wedge
a,z\wedge a)) \\
=&T_{2}(T_{2}(x\wedge a,y\wedge a),z\wedge a)=T(T_{2}(x\wedge a,y\wedge
a),z)\\
=&T(T(x,y),z).
\end{align*}

Therefore, we obtain that the binary operation $T:L^{2}\longrightarrow L$
defined by the formula (\ref{*}) is a t-norm on $L$.
\end{proof}

\medskip

We should note that Theorem~\ref{t-norm-Thm} exactly answers Question~\ref%
{Q-2}. We can also observe from Theorem~\ref{t-norm-Thm} that the ordinal
sum $T:L^{2}\longrightarrow L$ of two binary operations $T_{1}$ $%
:[a,1]^{2}\longrightarrow \lbrack a,1]$ and $T_{2}:[0,a]^{2}\longrightarrow
\lbrack 0,a]$ defined by the formula (\ref{*}) being a t-norm on a bounded
lattice $L$ is closely interested in the value of the binary operation $%
T_{2} $ on $(\left\{ a\right\} \cup I_{a})\times \{a\},$ while the binary
operation $T_{1}$ is inefficiency. Such an ordinal sum $T$ of two t-subnorms
$T_{1}$ and $T_{2}$ generates a t-norm on a bounded lattice $L$ under the
assumption (\ref{c}), whatever the values of $T_{2}$ and $T_{1}$ are on the
boundary $(\{a\}\times ([0,a]\setminus \{z\wedge a:z\in I_{a}\}))\cup
(([0,a]\setminus \{z\wedge a:z\in I_{a}\})\times \{a\})$ and $(\{1\}\times
\lbrack a,1])\cup( \lbrack a,a]\times \{1\})$, respectively.

Take in Theorem \ref{EY-Thm} the binary operation $T_{2}:[0,a]^{2}%
\longrightarrow \lbrack 0,a]$ defined by $T_{2}(x,y)=x\wedge y$ for all $%
x,y\in \lbrack 0,a].$ Then the ordinal sum operation with two summands given
by the formula (\ref{*}) reduces to the ordinal sum operation with one
summand introduced in \cite{EKM2015}, which is given in the following
Corollary \ref{c1}.

\begin{corollary}[\protect\cite{EKM2015}]
\label{c1} Let $(L,\leq ,0,1)$ be a bounded lattice and $a\in L\setminus
\{0,1\}$. If $T_{1}:[a,1]^{2}\longrightarrow \lbrack a,1]$ is a t-norm on $[a,1]$%
, then the binary operation $T^{(1)}:L^{2}\longrightarrow L$ defined by the
formula (\ref{T^1}) is a t-norm on $L$.
\begin{equation}
T^{(1)}(x,y)=%
\begin{cases}
x\wedge y, & x=1\text{ or }y=1, \\
T_{1}(x,y), & (x,y)\in \lbrack a,1)^{2}, \\
x\wedge y\wedge a, & \text{otherwise}.%
\end{cases}
\label{T^1}
\end{equation}
\end{corollary}

Take in Theorem \ref{EY-Thm} the binary operation $T_{2}:[0,a]^{2}%
\longrightarrow \lbrack 0,a]$ defined by $T_{2}(x,y)=x\wedge y$ for $a\in
\left\{ x,y\right\} $ and $T_{2}(x,y)=0$ for $x,y\in \left[ 0,a\right[ ^{2}.$
Then the ordinal sum operation with two summands given by the formula (\ref%
{*}) reduces to the ordinal sum operation with one summand introduced in
\cite{C2019}, which is given in the following Corollary \ref{c2}.

\begin{corollary}[\protect\cite{C2019}]
\label{c2} Let $(L,\leq ,0,1)$ be a bounded lattice and $a\in L\setminus
\{0,1\}$. If $T_{1}:[a,1]^{2}\longrightarrow \lbrack a,1]$ is a t-norm on $[a,1]$%
, then the binary operation $T^{(2)}:L^{2}\longrightarrow L$ defined by the
formula (\ref{T^2}) is a t-norm on $L$.
\begin{equation}
T^{(2)}(x,y)=%
\begin{cases}
x\wedge y, & x=1\text{ or }y=1, \\
0, & (x,y)\in \lbrack 0,a)^{2}\cup ([0,a)\times I_{a}) \\
&\quad \quad \cup
(I_{a}\times \lbrack 0,a))\cup (I_{a}\times I_{a}), \\
T_{1}(x,y), & (x,y)\in \lbrack a,1)^{2}, \\
x\wedge y\wedge a, & \text{otherwise}.%
\end{cases}
\label{T^2}
\end{equation}
\end{corollary}

In Theorem \ref{t-norm-Thm}, for the ordinal sum operation $%
T:L^{2}\longrightarrow L$ of two t-subnorms $T_{1}:[a,1]^{2}\longrightarrow \lbrack
a,1]$ and $T_{2}:[0,a]^{2}\longrightarrow \lbrack 0,a],$ which is given by
the formula (\ref{*}), generating a t-norm on a bounded lattice $L$, the
condition (\ref{c}) cannot be omitted, in general. In the following, we
provide an example of a bounded lattice $L,$ and a t-subnorm $%
T_{2}:[0,a]^{2}\longrightarrow \lbrack 0,a]$ violating the condition (\ref{c}%
) on which the ordinal sum $T$ defined by the formula (\ref{*}) is not a
t-norm on $L.$

\begin{example}
\label{Exa-3} {\rm Given the bounded lattice $L_{1}$ with Hasse diagram shown in
Fig.~\ref{Hasse-2}, it is clear that $I_{a}=\{c\},$ i.e., $I_{a}\neq
\emptyset .$ If we consider the binary operations $T_{1}:[a,1]^{2}%
\longrightarrow \lbrack a,1]$ and $T_{2}:[0,a]^{2}\longrightarrow \lbrack 0,a]$
defined by $T_{1}\left( x,y\right) =a$ for all $x,y\in \lbrack a,1]$ and $%
T_{2}\left( x,y\right) =0$ for all $x,y\in \lbrack 0,a]$, then both $T_{1}$
and $T_{2}$ are t-subnorms on $[a,1]$ and $[0,a]$, respectively. Notice that
from $T_{2}(c\wedge a,a)=0$ and $I_{a}\neq \emptyset ,$ the condition (\ref%
{c}) is violated. By applying the construction approach in the formula (\ref%
{*}), we obtain the ordinal sum operation $T:L_{1}\times L_{1}\longrightarrow
L_{1}$ given in Table~\ref{tab-b}. It is easily seen that $T$ is not
increasing on $L_{1}$ since $T(b,a)=b>0=T(c,a)$ for $b\leq c$. Therefore, $T$
is not a t-norm on $L_{1}$.
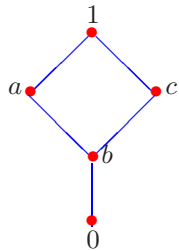
\begin{figure}[h]
\begin{center}
\begin{picture}(150,100)(100,-60)
\put(175.5, 25){$1$}
\put(175,19){\color{red}$\bullet$}
\put(176,20){\color{blue}\line(-1,-1){22}}
\put(152,-3.4){\color{red}$\bullet$}
\put(146,-2){$a$}
\put(179,20){\color{blue}\line(1,-1){22}}
\put(199,-3.4){\color{red}$\bullet$}
\put(205,-2){$c$}
\put(154,-2.5){\color{blue}\line(1,-1){24}}
\put(201,-2.5){\color{blue}\line(-1,-1){24}}
\put(175.5,-28){\color{red}$\bullet$}
\put(181,-28){$b$}
\put(177.5,-52){\color{blue}\line(0,0){24}}
\put(175,-52){\color{red}$\bullet$}
\put(175.5,-60){$0$}
\end{picture}
\end{center}
\caption{Hasse diagram of the lattice $L_{1}$ in Example~\protect\ref{Exa-3}}
\label{Hasse-2}
\end{figure}
\begin{table}[h]
\caption{The ordinal sum $T$ in Example~\protect\ref{Exa-3}}
\label{tab-b}\centering%
\begin{tabular}{|l|lllll|}
\hline
$T$ & $0$ & $b$ & $a$ & $c$ & $1$ \\ \hline
$0$ & $0$ & $0$ & $0$ & $0$ & $0$ \\
$b$ & $0$ & $0$ & $b$ & $0$ & $b$ \\
$a$ & $0$ & $b$ & $a$ & $0$ & $a$ \\
$c$ & $0$ & $0$ & $0$ & $0$ & $c$ \\
$1$ & $0$ & $b$ & $a$ & $c$ & $1$ \\ \hline
\end{tabular}%
\end{table}
}
\end{example}

In the following example, we consider a bounded lattice $L_{2},$ and two
t-subnorms $T_{1}$ and $T_{2}$ on the subintervals $[a,1]$ and $[0,a]$ of $%
L, $ respectively, which satisfy the condition (\ref{c}) in Theorem \ref%
{t-norm-Thm}. Taking into account of Theorem \ref{t-norm-Thm}, we observe that
the ordinal sum operation $T:L_{2}\times L_{2}\longrightarrow L_{2}$ with two
summands given by the formula (\ref{*}) is a t-norm on $L_{2}$.

\begin{example}
\label{Exa-2} {\rm Consider the bounded lattice $L_{2}$ with Hasse diagram shown
in Fig.~\ref{Hasse-1}. If we define the binary operations $%
T_{1}:[a,1]^{2}\longrightarrow \lbrack a,1]$ and $T_{2}:[0,a]^{2}\longrightarrow
\lbrack 0,a]$ by $T_{1}\left( x,y\right) =a$ for all $x,y\in \lbrack a,1]$
and $T_{2}\left( x,y\right) =0$ for all $x,y\in \lbrack 0,a]$, then neither $%
T_{1}$ nor $T_{2}$ is a t-norm on $[a,1]$ and $[0,a]$, respectively. $T_{1}$
and $T_{2}$ are t-subnorms on $[a,1]$ and $[0,a]$, respectively, as well as $%
T_{2}\left( a,b\wedge a\right) =b\wedge a=0$. That is, the condition (\ref{c}%
) is satisfied. By using the method in the fromula (\ref{*}), we define the
ordinal sum operation $T:L_{2}\times L_{2}\longrightarrow L_{2}$ as in Table~\ref%
{tab-a}. In view of Theorem \ref{t-norm-Thm}, we observe that $T$ is a
t-norm on $L_{2}.$
\begin{figure}[h]
\begin{center}
\begin{picture}(150,100)(100,-60)
\put(175.5, 26){$1$}
\put(175,20){{\color{red}$\bullet$}}
\put(176,20){\color{blue}\line(-1,-1){22}}
\put(152,-3.8){{\color{red}$\bullet$}}
\put(146,-2){$a$}
\put(179,20){\color{blue}\line(1,-1){22}}
\put(199,-3.8){{\color{red}$\bullet$}}
\put(205,-2){$b$}
\put(154,-2.5){\color{blue}\line(1,-1){24}}
\put(201,-2.5){\color{blue}\line(-1,-1){24}}
\put(175,-28){{\color{red}$\bullet$}}
\put(175.5,-36){$0$}
\end{picture}
\end{center}
\caption{Hasse diagram of the lattice $L$ in Example~\protect\ref{Exa-2}}
\label{Hasse-1}
\end{figure}
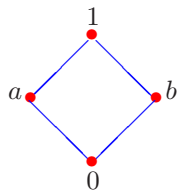
\begin{table}[h]
\caption{The ordinal sum $T$ in Example~\protect\ref{Exa-2}}
\label{tab-a}\centering
\begin{tabular}{|l|llll|}
\hline
$T$ & $0$ & $a$ & $b$ & $1$ \\ \hline
$0$ & $0$ & $0$ & $0$ & $0$ \\
$a$ & $0$ & $a$ & $0$ & $a$ \\
$b$ & $0$ & $0$ & $0$ & $b$ \\
$1$ & $0$ & $a$ & $b$ & $1$ \\ \hline
\end{tabular}%
\end{table}
}
\end{example}

\section{Concluding Remarks}\label{S-6}

This paper continues to study the ordinal sum operation with two summands on
a bounded lattice. Ertu\u{g}rul and Ye\c{s}ilyurt have
recently introduced an ordinal sum construction for generating t-norms on
the bounded lattice $L$ derived from two t-norms $T_{1}$ and $T_{2}$ on the
indicated subintervals $[a,1]$ and $[0,a]$ of $L$, respectively, for $a\in
L\setminus \{0,1\}$ \cite{EY2020}. Furthermore, they have put forward an open problem as
follows: if we take an associative, commutative, and monotone binary
operation instead of at least one of the t-norms on the subintervals of $L$,
will the same method work? By providing extra necessary and sufficient
conditions to ensure that an ordinal sum with two summands is a t-norm on a
bounded lattice $L$, we present a complete solution for their proposal. We
first give an example to show that the ordinal sum $T$ defined in Theorem~%
\ref{EY-Thm} may not be a t-norm on $L$ if we take an associative,
commutative, and monotone binary operation instead of at least one of the
t-norms $T_{1}$ and $T_{2}$. Moreover, in Theorem \ref{Increasingness-Thm},
given two commutative operations $T_{1}$ and $T_{2}$ on the subintervals $%
[a,1]$ and $[0,a]$ of $L$, respectively, we provide a sufficient and
necessary condition for the ordinal sum operation $T$ of $T_{1}$ and $T_{2}$
being increasing on $L$. Thanks to this theorem, we give a partial answer to
the mentioned problem. Then, in Theorem \ref{Increasingness-Thm}, we
introduce a sufficient and necessary condition for ensuring that the ordinal
sum operation $T$ of two t-subnorms $T_{1}$ and $T_{2}$ on the subintervals $%
[a,1]$ and $[0,a]$ of $L$ produces a t-norm on $L$. Through this theorem, we
answer exactly the above problem. For future work, it is interesting to
consider how to define the ordinal sum construction of t-subnorms on
subintervals of a bounded lattice being a t-norm without any additional
constraint.

\end{document}